\documentclass[a4paper,12pt]{article}
\usepackage[utf8x]{inputenc}
\usepackage{amsmath,amsthm,amssymb}
\usepackage[margin=3cm]{geometry}
\usepackage{dsfont}
\usepackage{authblk}

\title{Splitting the K-Terminal Reliability}
\author{Frank Simon\\ Email: {\tt simon@hs-mittweida}}
\affil{
Faculty Mathematics / Sciences / Computer Science\\ 
University Mittweida, Mittweida, Germany}

\newtheorem{Theorem}{Theorem}
\newtheorem{Lemma}[Theorem]{Lemma}
\newtheorem{Proposition}[Theorem]{Proposition}
\newtheorem{Corollary}[Theorem]{Corollary}
\theoremstyle{definition}
\newtheorem{Definition}[Theorem]{Definition}
\newtheorem{Example}[Theorem]{Example}
\newtheorem{Remark}[Theorem]{Remark}

\begin{document}
\maketitle
\thispagestyle{empty}
\begin{abstract}
Let $G=(V,E)$ be a graph and $K\subseteq V$ a set of terminal vertices.
Assume now that the edges of $G$ are failing independently with given
probabilities. The $K$-terminal reliability $R(G,K)$ is the probability that
all vertices in $K$ are mutually connected.

In this article we propose an efficient splitting formula for $R(G,K)$
at a separating vertex set of $G$ by lattice theoretic methods. \\[1ex]
\noindent
{\bf Keywords:} $K$-terminal reliability, Möbius inversion, partition lattice, join matrices, splitting 
\end{abstract}
\clearpage
\newcommand{\0}{\mathrm{0}}
\newcommand{\1}{\mathrm{1}}

\newcommand{\Real}{\mathds{R}}
\newcommand{\mon}{{\text{{\rm mon}}}}
\newcommand{\pr}{\mathrm{Pr}}

\newcommand{\F}{\mathbf{F}}
\newcommand{\E}{\mathbf{E}}
\newcommand{\G}{\mathbf{G}}
\newcommand{\M}{\mathbf{M}}

\newcommand{\Z}{\mathbf{Z}}
\newcommand{\T}{\mathbf{\Lambda}}
\newcommand{\rvec}{\mathbf{r}}
\newcommand{\pvec}{\mathbf{p}}

\newcommand{\fvec}{\mathbf{f}}
\newcommand{\gvec}{\mathbf{g}}

\section{Introduction}
Let $G=(V,E)$ be a graph and $K\subseteq V$ a set of terminal vertices. Assume now
that the edges of $G$ are failing independently with given probabilities. The $K$-terminal
reliability $R(G,K)$ is the probability that all terminal vertices are mutually connected in $G$.

Ball~\cite{Ball1980} shows that the computational complexity of $R(G,K)$ 
is $NP$-hard for arbitrary graphs. In the case of series parallel graphs Wood~\cite{Wood1985} proposed 
a polynomial time algorithm for the computation of $R(G,K)$ by polygon-to-chain reductions. 

A decomposition $(G^1,G^2,X)$ of $G$ consists of two subgraphs $G^1$ and $G^2$, so that
$G^1\cup G^2=G$ and $G^1\cap G^2=(X,\emptyset)$. Note that $X$ is a separating vertex set of $G$. 
In this article we propose a scheme for the computation of $R(G,K)$ given a decomposition, pursuing
the ideas of Rosenthal~\cite{Rosenthal1977}.

Bienstock~\cite{Bienstock1986} and Tittmann~\cite{Tittmann1999} 
examine such decomposition methods by utilising the lattice of set partitions of $X$. 
Nice results are especially derived when $K=V$ is assumed, but are unsatisfactory in the general case.

The centrepiece of this article is Theorem~\ref{thm:SymmetricSplittingFormula} representing 
$R(G,K)$ by the linear combination
\begin{align}\label{eqn:lk}
R(G,K)&=\sum_{\pi,\sigma\in \Pi_l(X,\pi_X)_0} R(G^1_\sigma, K_\sigma^1)f(\sigma,\pi)R(G^2_\pi,K_\pi^2),
\end{align}
where $G^1_\pi$ and $G^2_\sigma$ are emerging from the subgraphs $G^1$ and $G^2$ by a identification 
of vertices. Our result is therefore a generalisation of the result given by Bienstock~\cite{Bienstock1988}.

We emphasise that there are two main advantages of our approach compared to previously proposed methods for
the general $K$-terminal reliability by Rosenthal~\cite{Rosenthal1977} and Bienstock~\cite{Bienstock1986}.
  
The first advantage is the small cardinality of the state set $\Pi_l(X,\pi_X)_0$ in Equation~\ref{eqn:lk}. 
We show that if $X$ is a vertex separator of cardinality $n$, the state set can have at most $B(n+1)-1$ 
different elements. Here $B(n)$ is the $n$-th Bell number denoting the number of set partitions of an $n$-element set. 
We mention that the number of possible states might be even more reduced, if the separating vertex set 
contains terminal vertices.
Hence we are able to compute $R(G,K)$ even in the case of 
separating vertex sets, that were not accessible by former methods.

The second advantage of the new decomposition formula is the neat symmetry in its representation, which allows
a recursive application by the transfer-matrix method, which is not presented here for the sake of brevity.
         
\section{Partially Ordered Sets}
This section compiles some necessary definitions concerning partially ordered sets or short posets.
\begin{Definition}
Denote by $(P,\le_P)$ a poset of a finite set $P$. If $P$ has a maximum or minimum 
then it is denoted by $\1$ or $\0$, respectively. Given two elements $x,y\in P$, 
then $[x,y]_P:=\{z\in P:x\le z\le y\}$ is an \emph{interval} in $P$. Given any subset $Q\subseteq P$,
we say that $(Q,\le_Q)$ is a \emph{subposet} of $P$ if for all $p,q\in Q$ we have $p\le_Q q$ if and only
if $p\le_P q$. Observe that every interval $I$ of a poset $P$ is a subposet of $P$.
\end{Definition}
\begin{Definition}
Let $(P,\le_P)$ and $(Q,\le_Q)$ be two posets. 
The \emph{product order} $(P\times Q, \le_{P\times Q})$ consists of all ordered pairs 
in $P\times Q$, where $(p,q)\le_{P\times Q}(r,s)$ if and only if $p\le_P r$ and $q\le_Q s$. 
\end{Definition}

\begin{Definition}
Let $(P,\le_P)$ be a poset and $p,q\in P$. We say that $u\in P$ is an \emph{upper bound} of 
$p$ and $q$ if $p\le u$ and $q\le u$ and if every other upper bound $s\in P$ of $p$ and $q$ satisfies 
$u\le s$, we say that $u$ is the \emph{smallest upper bound} $u=p\vee q$ of $p$ and $q$. The notion
of \emph{lower bound} and the \emph{greatest lower bound} $p\wedge q$ of $p$ and $q$ is defined likewise.
\end{Definition}

\begin{Definition}
A \emph{lattice $(L,\vee,\wedge)$} is a poset $(L,\le)$, so that for all $p,q\in L$ the elements $p\vee q$ and 
$p\wedge q$ exist. In the case that we only demand that $p\vee q$ exists for all $p,q\in L$ we say that
$(L,\vee)$ is an \emph{upper semilattice}.
\end{Definition}

\begin{Definition}
Let $(P,\le_P)$ and $(Q,\le_Q)$ be two posets. 
A function $f\colon P\rightarrow Q$ is \emph{order preserving} if for all $p,q\in P$ 
we have $f(p)\le_Q f(q)$ if and only if $p\le_P q$. 
The posets $P$ and $Q$ are \emph{isomorphic}, if there is an
order preserving and bijective function $f\colon P\rightarrow Q$, and we write $P\simeq Q$.
\end{Definition}

\section{The Incidence Algebra\label{sec:incidenceAlgebra}}
This section states some of the definitions and results concerning
incidence algebras of posets.
Rota~\cite{Rota1964a} applies the incidence algebra of posets in combinatorics and Crapo~\cite{Crapo1966} 
contributes the versatile Theorem~\ref{thm:crapo}. Finally, we mention that Aigner~\cite{Aigner1997} 
gives a compilation of results, that are utilising incidence algebras in enumerative combinatorics. 

\begin{Definition}
Let $(P,\le)$ be a poset. We denote by $I(P)$ the set of all 
functions $f\colon P\times P\rightarrow \Real$ with $f(x,y)=0$, whenever $x\nleq y$ holds. 
For every $f,g\in I(P)$ define the \emph{convolution product} $f\star g\in I(P)$ by
\begin{align}\label{eqn:convolutionProduct}
(f\star g)(x,y)&=\sum_{x\le z\le y} f(x,z)g(z,y).
\end{align}
The set $I(P)$ endowed with the pointwise addition, multiplication with scalars $\lambda\in\Real$,
and the convolution product is the \emph{incidence algebra $I(P)$ of $P$}.
\end{Definition}

\begin{Definition}
Let $P$ be a poset and $x,y\in P$. The incidence functions 
\begin{align}
\zeta_P(x,y)=
\begin{cases}
1 & x\le y \\
0 & \text{else}              
\end{cases}
\quad \text{and} \quad
\delta_P(x,y)=
\begin{cases}
1 & x=y \\
0 & \text{else}
\end{cases}
\end{align}
are the \emph{Zeta-function} and the \emph{Delta-function} of $P$.
\end{Definition}

\begin{Definition}\label{def:moebFunc}
Let $P$ be a poset. The unique incidence function $\mu_P\in{}I(P)$, that satisfies the equation
$\mu_P\star \zeta_P=\delta_P$, is the \emph{Möbius function of $P$}.
\end{Definition}

\begin{Proposition}[Rota~\cite{Rota1964a}]\label{prop:productmoebius}
Let $(P,\le_P)$ and $(Q,\le_Q)$ be two posets. The Möbius function of the product order 
$(P\times Q, \le_{P\times Q})$ satisfies
\begin{align}
\mu_{P\times Q}((p,q),(r,s))&=\mu_P(p,r)\mu_Q(q,s)
\end{align}
for all $(p,q), (r,s)\in P\times Q$.
\end{Proposition}

\begin{Definition}
Let $L$ be a lattice with minimum $\0$ and maximum $\1$. $L$ is \emph{complemented} if
for all $p\in L$ there is a $q\in L$ with $p\vee q=\1$ and $p\wedge q=\0$.
\end{Definition}

\begin{Theorem}[Crapo \cite{Crapo1966}]\label{thm:crapo}
Let $L$ be a finite lattice, that is not complemented. Then $\mu_L(\0,\1)=0$.
\end{Theorem}
     
\section{Labelled Set Partitions\label{sec:labelledSetPartitions}}
In this section we introduce the lattice of labelled set partitions $\Pi_l(X)$ of a finite set $X$ 
and determine its Möbius function. The study of this lattice is helpful when considering the 
splitting of the $K$-terminal reliability. First approaches in this direction
are made by Bienstock~\cite{Bienstock1986} and Tittmann~\cite{Tittmann1990}.
\begin{Definition}
Let $X$ be a finite set. A \emph{set partition} $\pi=\{B_1,\ldots, B_k\}$ of $X$
is a collection of mutually disjoint and non-empty subsets of $X$, the \emph{blocks},
 with union $X$. The set of all set partitions of $X$ is denoted by $\Pi(X)$.

We define the poset $(\Pi(X),\le)$ by setting $\sigma\le\pi$ 
if every block of $\sigma$ is a subset of a block in $\pi$ for all $\sigma\in\pi\in\Pi(X)$.
Note that $(\Pi(X),\le)$ is 
a lattice with minimum $\hat{0}$ and maximum $\hat{1}$.

Finally, we mention that the number of all set partitions of an $n$-element set are 
the \emph{Bell numbers $B(n)$} and the number of all set partitions of an $n$-element set with
$k$ blocks the \emph{Stirling numbers of the second kind $S(n,k)$}. 
\end{Definition}

\begin{Theorem}[Rota~\cite{Rota1964a}]
\label{thm:moebiusPartitionLattice}
Let $X$ be a non-empty $n$-element set. Then the Möbius function in $\Pi(X)$ 
satisfies
\begin{align}
\mu_{\Pi(X)}(\hat{0},\hat{1})&=(-1)^{n-1}(n-1)!.
\end{align}
\end{Theorem}

\begin{Definition}
Let $X$ be a finite set and $l\not\in X$ a distinguished label element.
A {\em labelled set partition} $\pi$ is a collection of subsets
$\pi=\{B_1\cup L_1,\ldots, B_k\cup L_k\}$ with
$B_i\subseteq X$ and $L_i\in\{\emptyset,\{l\}\}$, so that $\{B_1,\ldots, B_k\}$ is a set
partition of $X$. The set of all labelled set partitions is denoted by
$\Pi_l(X)$.  A block $B_i\cup L_i\in \pi$ is {\em unlabelled} if $L_i=\emptyset$ 
and \emph{labelled} if $L_i=\{l\}$. 

For convenience of display we use the notation $\pi=B_1L_1|\ldots|B_kL_k$ and we drop
all unnecessary parentheses. For example we write $\pi=12l|3|45$ instead of
$\pi=\{\{1,2,l\},\{3\},\{4,5\}\}$.
\end{Definition}

\begin{Definition}
Let $\pi\in\Pi_l(X)$ with $\pi=B_1L_1|\ldots|B_kL_k$ and $Y\subseteq X$. 
We say that $\pi\sqcap Y\in\Pi_l(X)$ is the \emph{restriction of $\pi$ to $Y$} with
\begin{align}
\pi\sqcap Y&=\bigcup_{
\substack{
B_iL_i \in\pi \\
B_i\cap Y\ne\emptyset 
}
} \{(B_i\cap Y) \cup L_i)\}.
\end{align}
\end{Definition}

\begin{Definition}
Let $\sigma,\pi\in\Pi_l(X)$ and set
$\sigma\le\pi$ if every block of $\sigma$ is a subset of a block in
$\pi$. Observe that $(\Pi_l(X),\le)$ is a poset with minimum $\hat{0}_l$ and maximum $\hat{1}_l$.
It can be shown that $(\Pi_l(X),\le)$ is even more a lattice $(\Pi_l(X),\vee,\wedge)$.
\end{Definition}

\begin{Proposition}
\label{prop:product}
Let $\sigma,\pi\in\Pi_l(X)$ with $\sigma\le\pi$ and $\pi=B_1L_1|\ldots|B_kL_k$. 
Then the interval $[\sigma,\pi]_{\Pi_l(X)}$ is isomorphic to the $k$-fold product order
\begin{align}
\prod_{i=1}^k [\sigma\sqcap B_i,B_iL_i]_{\Pi_l(B_i)}.
\end{align}
\end{Proposition}

\begin{Lemma}
\label{lem:unevenlyLabelled}
Let $\pi\in\Pi_l(X)$ be a labelled set partition with at least one labelled and 
at least one unlabelled block. Then $\mu_{\Pi_l(X)}(\pi, \hat{1}_l)=0$.
\end{Lemma}
\begin{proof}
We can assume without loss of generality that $\pi=B_1L_1|B_2L_2|\ldots|B_kL_k$ with $k\ge 2$, so that $L_1=\emptyset$ and 
$L_2=\{l\}$. Consider now the labelled set partition $\tilde\pi\in\Pi_l(\{1,\ldots, k\})$ with
$\tilde\pi=1L_1|2L_2|\ldots |kL_k$. Then we have 
\begin{align*}
[\pi,\hat{1}_l]_{\Pi_l(X)}&\simeq[\tilde\pi,\hat{1}_l']_{\Pi_l(\{1,\ldots, k\})},
\end{align*}
where $\hat{1}_l'$ denotes the maximum in $\Pi_l(\{1,\ldots, k\})$ and hence
\begin{align*}
\mu_{\Pi_l(X)}(\pi,\hat{1}_l)&=\mu_{\Pi_l(\{1,\ldots,k\})}(\tilde\pi,\hat{1}_l').
\end{align*}

Therefore we can assume without loss of generality that $\pi$ has the form $\pi=1L_1|2L_2|\ldots|kL_k$.

Define the labelled set partition $\pi'=1L_1'|2L_2|\ldots|kL_k$, where
$L_1'=L_1\cup \{l\}$ is set, and observe that $\pi'$ is
an element of the interval $I:=[\pi,\hat{1}_l]_{\Pi_l(X)}$ with $\pi'>\pi$. 

Suppose now that $\pi'$ has the complement $\sigma\in{}I$, then
$\pi'\vee\sigma=\hat{1}_l$ implies $\sigma=\hat{1}_l$.
On the other hand we find $\sigma\wedge \pi'=\pi'>\pi$, so that $\pi'$
has no complement in the interval $I$, which contradicts our assumption.
Thus we can conclude by Crapo's Theorem~\ref{thm:crapo} 
that $\mu_{\Pi_l(X)}(\pi,\hat{1}_l)=0$.
\end{proof}

\begin{Definition}\label{def:moebValues}
Let $X=\{1,\ldots, n\}$ and define
\begin{align}
\mu_n&=\mu_{\Pi_l(X)}(\hat{0}_l,\hat{1}_l) \\
\tilde\mu_n&=\mu_{\Pi_l(X)}(\sigma_n,\hat{1}_l),
\end{align}
where $\sigma_n=1l|\ldots|nl$ denotes the labelled set partition with $n$ labelled singleton blocks.
\end{Definition}

\begin{Example}\label{example:productOrder}
Let $X=\{1,2,3,4,5,6,7\}$ and $\sigma,\pi\in\Pi_l(X)$ with $\sigma=1l|2l|34|5|67$ and $\pi=12l|345l|67$. Then we find
by Propositions~\ref{prop:product} and \ref{prop:productmoebius} and Definition~\ref{def:moebValues}
\begin{align*}
[\sigma,\pi]_{\Pi_l(X)}&\simeq [1l|2l,12l]_{\Pi_l(\{1,2\})}\times[34|5,345l]_{\Pi_l(\{3,4,5\})}\times [67,67]_{\Pi_l(\{6,7\})} \\
&\simeq [1l|2l,12l]_{\Pi_l(\{1,2\})}\times [3|5,35l]_{\Pi_l(\{3,5\})} \times [6l,6l]_{\Pi_l(\{6\})}.
\end{align*}
Hence we have $\mu_{\Pi_l(X)}(\sigma,\pi)=\tilde\mu_2\mu_2\tilde\mu_1$.
\end{Example}

By Example~\ref{example:productOrder} we conclude, that we only have to consider $\mu_n$ and $\tilde\mu_n$ to compute the Möbius function
$\mu_{\Pi_l(X)}(\sigma,\pi)$ for arbitrary $\sigma,\pi\in\Pi_l(X)$.

\begin{Theorem}\label{thm:keyTheoremMoebiusFunction}
Let $X$ be a non-empty $n$-element set. Then
\begin{align}
\tilde\mu_n&=(-1)^{n-1}(n-1)!,  \\
\mu_n&=(-1)^n (n-1)!
\end{align}
for all $n\ge 1$.
\end{Theorem}
\begin{proof}
Observe that the interval $[\sigma_n,\hat{1}_l]_{\Pi_l(X)}$ is isomorphic to the interval
$[\hat{0},\hat{1}]_{\Pi(X)}$ in the partition lattice $\Pi(X)$. Hence we can conclude by
Theorem~\ref{thm:moebiusPartitionLattice}
\begin{align*}
\tilde \mu_n&=\mu_{\Pi(X)}(\hat{0},\hat{1})=(-1)^{n-1}(n-1)!,
\end{align*}
which proves the first claim.

Now the Möbius function satisfies by Definition~\ref{def:moebFunc}
\begin{align*}
\sum_{\pi\in\Pi_l(X)} \mu_{\Pi_l(X)}(\pi,\hat{1}_l)&=\delta_{\Pi_l(X)}(\hat{0}_l,\hat{1}_l)=0,
\end{align*}
with $\delta_{\Pi_l(X)}(\hat{0}_l,\hat{1}_l)=0$, as $\hat{0}_l\ne\hat{1}_l$, whenever $X$ is a non-empty set. 
The application of  Lemma~\ref{lem:unevenlyLabelled} allows the reduction of the above sum to 
the non-vanishing Möbius function values and we have
\begin{align*}
\sum_{k=1}^n S(n,k)\left(\mu_k+\tilde\mu_k\right)&=0.
\end{align*}
Here we used the property, that every interval $[\pi,\hat{1}_l]_{\Pi_l(X)}$ with  
$\pi=B_1L_1|\ldots|B_kL_k$ is isomorphic to the interval $[1L_1|\ldots|kL_k,\hat{1}_l]_{\Pi_l(\{1,\ldots,k\})}$. 

Subsequently we show by induction over $n$ that $\mu_n=-\tilde\mu_n$. For the basic step $n=1$ the claim is
obviously true. Let $l=n+1$ and assume that the claim is true for all $l\le n$. Then we have
after application of the induction hypothesis to the above sum
\begin{align*}
\sum_{k=1}^{n+1} S(n+1,k)\left(\mu_k+\tilde\mu_k\right)&=0 \\
S(n+1,n+1)\left(\mu_{n+1}+\tilde\mu_{n+1}\right)&=0 \\
\mu_{n+1}&=-\tilde\mu_{n+1},
\end{align*}
which proves the claim for $l=n+1$ as well. Hence we have
\begin{align*}
\mu_n&=-\tilde\mu_n=-1\cdot(-1)^{n-1}(n-1)!
\end{align*}
for all $n\ge 1$.
\end{proof}
 
\section{The K-Terminal Reliability\label{sec:kterminalreliability}}
In this section we derive a first splitting approach for the $K$-terminal reliability $R(G,K)$ at a
separating vertex set $X$ in Theorem~\ref{thm:PsplittingFormula}.
Furthermore, we examine $R(G,K)$ by defining
suitable indicator functions following the ideas presented in the PhD thesis of 
Tittmann~\cite{Tittmann1990}.
\begin{Definition}
A {\em graph} $G=(V,E,\varphi)$ is a triple consisting of a finite set
$V$ of {\em vertices} and a finite set $E$ of {\em edges} endowed
with an incidence function $\varphi\colon E\rightarrow
V^{(2)}$. Here $V^{(2)}$ denotes the set of the two-element subsets of $V$. 

For convenience of display we often tacitly omit the 
incidence function $\varphi$ and just write $G=(V,E)$. 
Let $F\subseteq E$, then $H=(V,F)$ is a \emph{spanning subgraph} of $G$ 
and we write $H\subseteq G$.
\end{Definition}

\begin{Definition}
Let $G=(V,E,\varphi)$ be a graph and $\pi=B_1|\ldots|B_k\in\Pi(X)$ with
$X\subseteq V$. Then $G_\pi=(V_\pi,E_\pi,\varphi_\pi)$ denotes the 
\emph{$\pi$-merging of $G$} with
vertex set 
\begin{align}
V_\pi&=(V\setminus X)\cup \{B_1,\ldots, B_k\}
\end{align}
and edge set
\begin{align}
E_\pi&=\{e\in E:\varphi(e) \text{ is not a subset of a block in } \pi \},
\end{align}
where the incidence function $\varphi_\pi\colon E_\pi\rightarrow V_\pi^{(2)}$ is
given by
\begin{align}
\varphi_\pi(e)&=
\begin{cases}
	\{B_i, B_j\} & \text{if } \varphi(e)=\{u,v\}, u\in B_i, v\in B_j, B_i\ne B_j \\
	\{B_i, v\}   &  \text{if } \varphi(e)=\{u,v\}, u\in B_i, v\in V\setminus X \\
	\varphi(e) & \text{else}.
\end{cases}
\end{align}
In other words, $G_\pi$ denotes the graph that emerges from $G$ by merging all
vertices in $G$ that are in a same block in $\pi$, where possibly occurring
parallel edges are kept and loops are removed. In the case of the one block set
partition $\pi=\{X\}=\hat{1}$, we simply write $G_X=(V_X,E_X)$.
\end{Definition}

\begin{Definition}
Let $G=(V,E)$ be a graph and $p\colon{}E\rightarrow [0,1], e\mapsto p(e)$. 
Assume now that the edges $e\in E$ of $G$ are failing
independently with the probabilities $q(e):=1-p(e)$. We say that the pair $(G,p)$ 
is a {\em stochastic network}. 
In the following we identify the graph $G$ and its corresponding stochastic network
$(G,p)$ if there is no danger of confusion. The probability that the spanning subgraph
$H=(V,F)\subseteq G$ is realised equals
\begin{align}
\pr(H)&=\prod_{e\in F} p(e) \prod_{e\in E\setminus F} q(e).
\end{align}
\end{Definition}

\begin{Definition}
A \emph{$K$-graph $(G,K)$} is a pair consisting of a graph $G=(V,E)$ and a subset $K\subseteq V$
of \emph{terminal vertices} with $|K|\ge 2$.
Every $K$-graph $(G,K)$ induces a labelled set partition
\begin{align}
\{(G,K)\}&=V_1 L_1|\ldots|V_rL_r\in \Pi_l(V),
\end{align}
where two vertices $u$, $v\in V$ are in a same $V_i$ if and only if 
$u$ and $v$ are connected in $G$ and we set $L_i=\{l\}$ if $V_i\cap K\ne\emptyset$ 
and $L_i=\emptyset$ otherwise.

A graph $G=(V,E)$ is {\em $K$-connected} if all terminal vertices in $K$ are mutually 
connected in $G$ or in other words $\{(G,K)\}$ has exactly one labelled block. 
Finally, we define the {\em $K$-connectedness indicator $M(G, K)$} by 
\begin{align}
M(G,K)&=
\begin{cases}
1 & \text{$G$ is $K$-connected} \\
0 & \text{else}.
\end{cases}
\end{align}
\end{Definition}

\begin{Definition}\label{def:kterminalReliability}
Let $(G,K)$ be a $K$-graph and $(G,p)$ a stochastic network. We say that
$(G,K)$ is a {\em $K$-network} and the {\em $K$-terminal reliability
$R(G,K)$} is the probability that $G$ is $K$-connected. Thus
\begin{align}
R(G,K)&=\sum_{H\subseteq G}M(H,K)\pr(H).
\end{align}
\end{Definition}

\begin{Definition}
Let $(G,K)$ be a $K$-graph and $G^1=(V^1,E^1)$, $G^2=(V^2,E^2)$ subgraphs of $G$, so
that $E^1\cup E^2=E$, $E^1\cap E^2=\emptyset$, $V^1\cup V^2=V$ and $V^1\cap
V^2=X$ holds. Furthermore we demand that $K^1:=V^1\cap K\ne\emptyset$ and
$K^2:=V^2\cap K\ne \emptyset$ is satisfied,
so that each of the two subgraphs $G^1$ and $G^2$ contains at least one terminal vertex.

Under these two conditions $(G^1,G^2,X)$ is said to be a \emph{$K$-splitting of $(G,K)$} 
with \emph{separating vertex set $X$}. 
\end{Definition}

\begin{Definition}
Define the function $m\colon \Pi_l(X)\rightarrow\{0,1\},\pi \mapsto
m(\pi)$ by
\begin{align}
m(\pi)&=
\begin{cases}
1 & \pi \text{ has exactly one labelled block} \\
0 & \text{else}.
\end{cases}
\end{align}
\end{Definition}

\begin{Definition}
Let $(G^1,G^2,X)$ be a $K$-splitting of $(G,K)$ and define
\begin{align}
D(G^i, \pi)&=
\begin{cases}
1 & \{(G^i,K^i)\}\sqcap X=\pi \\
0 & \text{ else}
\end{cases}
\end{align}
for all $\pi\in\Pi_l(X)$ and $i=1,2$.
\end{Definition}

\begin{Definition}
Let $(G^1,G^2,X)$ be a $K$-splitting of $(G,K)$. Define for $i=1,2$ the $K$-graphs
$(G_X^i,K_X^i)$ with terminal vertex set $K_X^i=(K^i\setminus X)\cup \{X\}$.
\end{Definition}

\begin{Remark}\label{rem:implication}
Let $(G^1,G^2,X)$ be a $K$-splitting of $(G,K)$.
Observe that the $K$-connectedness of $G$ implies that every terminal
vertex $v\in K^i$ in $G^i$ is connected to at least one vertex in $X$ in $G^i$ for $i=1,2$.
Note that the above implication can be restated by using the $K$-connectedness indicator as
\begin{align}
M(G,K)=1\quad \Longrightarrow\quad M(G_X^1, K_X^1)M(G_X^2,K_X^2)=1.
\end{align}
\end{Remark}

\begin{Definition}
Let $(G^1,G^2,X)$ be a $K$-splitting of $(G,K)$ and $\pi\in\Pi_l(X)$.
The {\em partition probability $P(G^i,\pi)$} is defined as
\begin{align}
P(G^i,\pi)&=\sum_{H^i\subseteq G^i} D(H^i,\pi)M(H_X^i,K_X^i) \pr(H^i)
\qquad \text{ for } i=1,2.
\end{align}
\end{Definition}

\begin{Definition}
Let $\pi=B_1L_1|\ldots|B_kL_k\in\Pi_l(X)$ and $(G^1,G^2,X)$ be a $K$-splitting of $(G,K)$. 
We say that $(G_\pi^i,K_\pi^i)$ for $i=1,2$ is the \emph{$\pi$-merged $K$-graph of $(G^i,K^i)$}, where
$G_\pi^i$ denotes the $\pi'$-merging of $G^i$ with $\pi'=B_1|\ldots|B_k\in\Pi(X)$ and $K_{\pi}^i$ the 
terminal vertex set
\begin{align}
K_\pi^i&=(K^i\setminus X) \cup \{B_i\in\pi'\colon L_i=\{l\} \text{ or } B_i\cap K^i\ne\emptyset\}.
\end{align} 
\end{Definition}

\begin{Definition}
Let $X=\{x_1,\ldots, x_n\}$ and $(G^1,G^2,X)$ a $K$-splitting of $(G,K)$.
Denote by $\pi_X\in\Pi_l(X)$ the labelled 
set partition $\pi_X=x_1L_1|\ldots |x_nL_n$ with $L_i=\{l\}$ whenever $x_i\in K$ and
$L_i=\emptyset$ otherwise. 

Furthermore, let $\Pi_l(X,\pi_X)$ be the set of all labelled 
set partitions $\pi\ge\pi_X$ in $\Pi_l(X)$ with at least one labelled block.

We denote by $P(n,k)$ the number of elements in $\Pi_l(X, \pi_X)$, where we assume that 
$\pi_X$ has $k$ labelled and $n-k$ unlabelled blocks.
\end{Definition}

\begin{Theorem}
The numbers $P(n,k)$ are given by
\begin{align*}
P(n,0)&=\sum_{j=1}^{n}\binom{n}{j}B(j)B(n-j),  \\
P(n,k)&=\sum_{j=0}^{n-k}\binom{n-k}{j}B(k+j)B(n-k-j) 
\end{align*}
for $n\ge 1$ and $k\ge 1$.
\end{Theorem}
\begin{proof}
Let $X$ be a non-empty $n$-element set and consider the labelled set partition $\pi_X$ with $k$ labelled and $n-k$ 
unlabelled blocks and the number of possible ways to construct a labelled set partition $\sigma\ge \pi_X$ 
with at least one labelled block. 

We can choose in $\binom{n-k}{j}$ different ways $j$ of the $n-k$ unlabelled blocks of $\pi_X$ being 
labelled in $\sigma$.

Afterwards we have $B(k+j)$ possibilities to partition the labelled blocks and $B(n-k-j)$ choices to
partition the remaining unlabelled blocks. In the case $k\ge 1$ there is always at least
one labelled block. For $k=0$ we ensure the existence of at least one labelled block in $\sigma$
by demanding $j\ge 1$.
\end{proof}

\begin{Theorem}\label{thm:kterminalSplitting}
Let $(G^1,G^2,X)$ be a $K$-splitting of $(G,K)$. Then
\begin{multline}
M(G,K)=\\
M(G^1_X,K^{1}_X)
M(G^2_X,K^{2}_X)
\sum_{\sigma_1,\sigma_2\in\Pi_l(X,\pi_X)}
D(G^1,\sigma_1)m(\sigma_1\vee\sigma_2)D(G^2,\sigma_2).
\end{multline}
\end{Theorem}
\begin{proof}
Assume first that there exists a terminal vertex $v\in K$, that is not connected
to a vertex in $X$. In this case we have 
\begin{align}
M(G^1_X,K^1_X)M(G^2_X,K^2_X)=0
\end{align}
and by the contraposition of Remark~\ref{rem:implication} we find $M(G,K)=0$
as well, so that the equation is valid in this trivial case.

Hence we can assume now that every terminal vertex $v\in K$ is connected to at least
one vertex in $X$ or in other words $M(G^1_X,K^1_X)M(G^2_X,K^2_X)=1$.

As $K^1$ and $K^2$ are non-empty sets, we can ensure
that every subgraph $G^i$ induces exactly one labelled set partition $\rho_i$ in $X$,
which has at least one labelled block and we conclude that $\rho_i\in \Pi_l(X,\pi_X)$.
Therefore the equation simplifies to
\begin{align*}
M(G,K)&=m(\rho_1\vee\rho_2).
\end{align*}
Observe now that under the above assumptions $G$ is $K$-connected if and only if for every 
two vertices $u,v\in X$, that are connected to a labelled vertex in $G$, there is a sequence 
of blocks $B_1,\ldots, B_k\in \rho_1\cup\rho_2$ with $B_i\cap B_{i+1}\ne\emptyset$ for 
$i=1,\ldots, k-1$, so that $u\in B_1$,$v\in B_k$.

This last characterisation is equivalent to the condition 
that $\rho_1\vee\rho_2$ has exactly one labelled block or in other words $m(\rho_1\vee\rho_2)=1$.
\end{proof}

\begin{Definition}
Let $(G^1,G^2,X)$ be a $K$-splitting of $(G,K)$.
Denote by $\pvec(G^i)$ and $\rvec(G^i)$ the vectors containing the probabilities
$P(G^i,\pi)$ and $R(G^i_\pi,K_\pi^i)$ for all $\pi\in\Pi_l(X,\pi_X)$ and 
define the \emph{transfer matrix} $\M$ as
\begin{align}
\M=(m(\pi\vee\sigma))_{\substack{\pi\in\Pi_l(X,\pi_X)\\\sigma\in\Pi_l(X,\pi_X)}}.
\end{align}
\end{Definition}

\begin{Theorem}\label{thm:PsplittingFormula}
Let $(G^1,G^2,X)$ be a $K$-splitting of $(G,K)$. Then
\begin{align}
R(G,K)&=\pvec(G^1)^T\M\pvec(G^2).
\end{align}
\end{Theorem}
\begin{proof}
The $K$-terminal reliability $R(G,K)$ is by Definition~\ref{def:kterminalReliability}
\begin{align*}
R(G,K)&=\sum_{H\subseteq G}M(H,K)\pr(H).
\end{align*}
The application of Theorem~\ref{thm:kterminalSplitting} to $M(H,K)$
and the definition of the partition probability yields
\begin{align*}
R(G,K)=
\sum_{
\sigma_1,\sigma_2\in\Pi_l(X,\pi_X)
}
P(G^1,\sigma_1)m(\sigma_1\vee\sigma_2)P(G^2,\sigma_2),
\end{align*}
which equals the stated matrix equation.
\end{proof}

In Theorem~\ref{thm:rplemma} we give a slight generalisation of a theorem found in the 
PhD thesis of Tittmann~\cite{Tittmann1990}.
\begin{Theorem}
\label{thm:rplemma}
Let $(G^1,G^2,X)$ be a $K$-splitting of $(G,K)$. Then
\begin{align}
M(G^i_\pi,K_\pi^i)&=
M(G_X^i,K_X^i)\sum_{\sigma\in\Pi_l(X,\pi_X)}
D(G^i, \sigma)m(\pi\vee\sigma)
\end{align}
holds for all $\pi\in\Pi_l(X,\pi_X)$ and $i=1,2$.
\end{Theorem}
\begin{proof}
First assume that there is a terminal vertex $v\in K^i$, which is not connected to 
a vertex in $X$. In this case we have $M(G_X^i,K_X^i)=0$ and $M(G^i_\pi,K_\pi^i)=0$
as well, as the graph $G^i_\pi$ has at least one 
terminal vertex in $X$, because $\pi\in\Pi_l(X,\pi_X)$ has at at least one labelled block. 
Therefore the equation is valid in this trivial case.

Hence we can assume from now on, that every terminal vertex $v\in K^i$ is connected
to a vertex in $X$ or in other words $M(G_X^i,K_X^i)=1$. As $K^i\ne\emptyset$ we
conclude that $G^i$ induces in $X$ exactly one labelled set partition $\rho$
with at least one labelled block, that satisfies $\rho\ge \pi_X$. This leaves us
with one summand
\begin{align}
M(G^i_\pi,K_\pi^i)&=m(\pi\vee\rho),
\end{align}
which is a valid equation by considering the properties of $\pi\vee\rho$ and $(G_\pi^i,K_\pi^i)$.
\end{proof}

\begin{Theorem}\label{thm:cformula}
Let $(G^1,G^2,X)$ be a $K$-terminal splitting of $(G,K)$. Then
\begin{align}
\rvec(G^i)&=\M\pvec(G^i) 
\end{align}
holds for $i=1,2$.
\end{Theorem}
\begin{proof}
Let $\pi\in\Pi_l(X,\pi_X)$ and consider the row of the above matrix
equation corresponding to $\pi$
\begin{align*}
R(G_\pi^i,K_\pi^i)&=
\sum_{\sigma\in\Pi_l(X,\pi_X)}
P(G^i,\sigma)m(\pi\vee\sigma).
\end{align*}
Observe that this equation holds by Theorem~\ref{thm:rplemma}, when we consider
the definition of the partition probability and a summation over all possible
subgraphs.
\end{proof}
 
\section{The Transfer Matrix\label{sec:transferMatrix}}
In his PhD thesis Tittmann~\cite{Tittmann1990} observed that the transfer matrix is generally 
not invertible.
This section states a factorisation of the $\M$ matrix and gives a condition for the existence
of $\M^{-1}$ by Corollary~\ref{cor:invertibility}.
This factorisation is then used in the computation of the $K$-terminal reliability
in Section~\ref{sec:splittingFormula}. 
Factorisations of supremum matrices are considered by Wilf~\cite{Wilf1968}, Smith~\cite{Smith1875}
and Lindström~\cite{Lindstrom1969} to solve problems in combinatorics and number theory. 
Haukkanen and Korkee~\cite{Korkee2008} give further remarks on determinants and 
inverses of supremum matrices. A nice introduction into the versatile Möbius inversion principle is given by Bender and
Goldmann~\cite{Bender1975}.

\begin{Theorem}[Wilf~\cite{Wilf1968}]\label{thm:wilf}
Let $P$ be a finite upper semilattice with $P=\{p_1,\ldots, p_n\}$, so that $p_i\le p_j$ implies $i\le j$ and
define the upper triangular Zeta-matrix $\E=(e_{ij})$ of format $n\times n$ with entries $e_{ij}=\zeta_P(p_i,p_j)$.

Furthermore let $f,g\colon P\rightarrow \Real$ be two functions defining the vectors $\fvec=(f(p_i))$ and 
$\gvec=(g(p_i))$ of length $n$, so that 
\begin{align*}
\fvec=\E \gvec 
\end{align*}
is satisfied.

Now denote by $\mathbf{F}=(f_{ij})$ and $\mathbf{G}=(g_{ij})$ matrices of format $n\times n$
with $f_{ij}=f(p_i\vee p_j)$ and $\mathbf{G}$ being a diagonal matrix with entries $g_{ii}=g(p_i)$. Then
\begin{align}
\mathbf{F}&=
\mathbf{E}\mathbf{G}\mathbf{E}^T.
\end{align}
\end{Theorem}

\begin{Remark}\label{remark:moebInv}
Assume that the conditions of Theorem~\ref{thm:wilf} are given. We can then compute
the vector $\gvec$ from the vector $\fvec$ by
\begin{align}
\gvec&=\E^{-1}\fvec{},
\end{align}
which is the \emph{Möbius inversion principle}.
Observe that the entries $r_{ij}$ of the matrix $\E^{-1}=(r_{ij})$ satisfy $r_{ij}=\mu_P(p_i,p_j)$.
\end{Remark}

\begin{Definition}\label{def:0selection}
Assume that the conditions of Theorem~\ref{thm:wilf} are satisfied and 
define the subset $P_0\subseteq P$ by
\begin{align}
P_0&=\{p\in P\colon g(p)\ne 0\}.
\end{align}
Furthermore denote by $\F_0$, $\E_0$ and $\G_0$ the matrices emerging form $\F$, $\E$ and $\G$ by the removal of all 
columns and rows, that are not corresponding to elements in $P_0$. 
In general we denote by the bracket notation $[\cdot]_0$ the removal of all 
rows and columns not corresponding to elements in $P_0$.
\end{Definition}

\begin{Theorem}\label{thm:inverseWilf}
Assume that the conditions of Theorem~\ref{thm:wilf} are satisfied. 
Then the inverse of $\F_0$ exists.
\end{Theorem}

\begin{proof}
By Theorem~\ref{thm:wilf} we have $\F=\E\G\E^T$ and therefore 
\begin{align*}
f_{ij}&=\sum_{k=1}^n e_{ik}g_{kk}e_{jk}=
\sum_{
k:p_k\in P_0
} e_{ik}g_{kk}e_{jk},
\end{align*}
for all $p_i,p_j\in P_0$, which gives $\F_0=\E_0\G_0\E_0^T$. 
Observe now that the inverses of $\E_0$ and $\G_0$ exist, as $\E_0$ represents the Zeta-function of the 
subposet $P_0$ and $\G_0$ is a diagonal matrix with non-vanishing diagonal elements. 
\end{proof}

\begin{Remark}
\label{def:thetafunction}
First observe that the set $\Pi_l(X,\pi_X)$ is a finite upper semilattice.
In the following we apply Theorem~\ref{thm:inverseWilf} to the transfer matrix $\M$ and the 
function $g$ in Theorem~\ref{thm:wilf} is denoted by $\lambda$. Hence we have 
by Remark~\ref{remark:moebInv}
\begin{align}
\lambda(\pi)&=\sum_{
\substack{
\sigma\in\Pi_l(X,\pi_X)\\
\sigma\ge\pi
}
} \mu_{\Pi_l(X)}(\pi,\sigma)m(\sigma)
\end{align}
for all $\pi\in\Pi_l(X,\pi_X)$.
\end{Remark}

\begin{Definition}
Let $(G^1,G^2,X)$ be a $K$-splitting of $(G,K)$.
Denote by $\T$ the diagonal matrix with entries $\lambda_{\pi,\pi}=\lambda(\pi)$ 
for all $\pi\in\Pi_l(X,\pi_X)$ and by $\Z$ the matrix with entries $z_{\pi,\sigma}=[\pi\le\sigma]$ 
for all $\pi,\sigma\in\Pi_l(X,\pi_X)$.
\label{def:thetaMatrix}
\label{def:zetaMatrix}
\end{Definition}

\begin{Corollary}\label{cor:factoringM}
We have
\begin{align*}
\M&=\Z\T\Z^T \quad \text{ and } \quad \M_0=\Z_0\T_0\Z_0^T,
\end{align*}
where $\M_0^{-1}$ always exists. 
\end{Corollary}

\begin{Definition}
Let $X$ be an $n$-element set and assume that $\pi_X$ has $k$ labelled and $n-k$ unlabelled blocks.
Denote now by $\Pi_l(X,\pi_X)_0$ the set
\begin{align}
\Pi_l(X,\pi_X)_0&=\{\pi\in\Pi_l(X,\pi_X):\lambda(\pi)\ne 0\}
\end{align}
and the number of elements in $\Pi_l(X,\pi_X)_0$ by $P_0(n,k)$.
\end{Definition}

\begin{Corollary}\label{cor:invertibility}
The matrix $\M$ is invertible if and only if $P(n,k)=P_0(n,k)$ is satisfied. In this case we have $\M=\M_0$.
\end{Corollary}

\begin{Definition}
Let $\pi\in\Pi_l(X,\pi_X)$ and denote by $\pi^*\in\Pi_l(X,\pi_X)$ the smallest labelled set partition with at most 
one labelled block, so that $\pi^*\ge\pi$. In other words $\pi^*$ emerges from $\pi$ by the union of all
labelled blocks in $\pi$.
\end{Definition}

\begin{Theorem}
The set $\Pi(X,\pi_X)_0$ consists of all labelled set partitions $\pi$ in $\Pi_l(X,\pi_X)$ with at most one unlabelled block.
We have even more  
\begin{align}
\lambda(\pi)&=\mu_{\Pi_l(X)}(\pi, \pi^*)
\end{align}
for all $\pi\in\Pi(X,\pi_X)_0$.
\end{Theorem}
\begin{proof}
Partition $\pi$ by setting $\pi=\rho\cup\tau$, where $\rho\in\Pi_l(L)$ and $\tau\in\Pi(U)$ 
are consisting of all labelled and all unlabelled blocks of $\pi$, respectively. 
Hence we have $L\cup U=X$ and $L\cap U=\emptyset$. It is by definition
\begin{align*}
	\lambda(\pi)&=\sum_{\sigma\in\Pi_l(X)}\mu(\pi,\sigma)m(\sigma).
\end{align*}
Now the condition $\mu_{\Pi_l(X)}(\pi,\sigma)\ne 0$ implies by Lemma~\ref{lem:unevenlyLabelled} that $\sigma\ge \pi$
has the form $\sigma=\varepsilon\cup\omega$ where $\varepsilon\in\Pi_l(L)$ and 
$\tau\in\Pi(U)$.

Therefore we can write
\begin{align*}
	\lambda(\pi)&=\sum_{
	\substack{
		\varepsilon\in\Pi_l(L)\\
		\omega\in\Pi(U)
	}}
	\mu_{\Pi_l(X)}(\rho\cup\tau,\varepsilon\cup\omega)m(\varepsilon\cup\omega) \\
	&=\sum_{\varepsilon\in\Pi_l(L)}
	\sum_{\omega\in\Pi(U)}
	\mu_{\Pi_l(L)}(\rho,\varepsilon)
	\mu_{\Pi(U)}(\tau,\omega)m(\varepsilon\cup\omega),
\end{align*}
where the last line follows by the application of Proposition~\ref{prop:productmoebius}. Note that 
$m(\varepsilon\cup\omega)=1$  if and only if $\varepsilon=\rho^*$ for all $\varepsilon\in\Pi_l(L)$ and $\omega\in\Pi(U)$.

Hence
\begin{align*}
	\lambda(\pi)&=
	\mu_{\Pi_l(L)}(\rho,\rho^*)
	\sum_{\omega\in\Pi(U)}
	\mu_{\Pi(U)}(\tau,\omega) \\
	&=\mu_{\Pi_l(L)}(\rho,\rho^*)\delta_{\Pi(U)}(\tau,\hat{1}),
\end{align*}
where the last line follows from the definition of the Möbius function. This 
proves the claim, as $\mu_{\Pi_l(L)}(\rho,\rho^*)$ is
due to Theorem~\ref{thm:keyTheoremMoebiusFunction} non-vanishing.
\end{proof}

\begin{Theorem}
The numbers $P_0(n,k)$ are given by
\begin{align}
P_0(n,0)&=\sum_{j=0}^{n-1}\binom{n}{j}B(n-j)=B(n+1)-1, \\
P_0(n,k)&=\sum_{j=0}^{n-k}\binom{n-k}{j}B(n-j)
\end{align}
for $n,k\ge 1$.
\end{Theorem}
\begin{proof}
Assume that $\pi_X$ is a labelled set partition with $k$ labelled and $n-k$ unlabelled blocks. 
We count all labelled set partitions $\sigma\ge \pi_X$ with at most one unlabelled block 
and at least one labelled block.

Now we can choose in 
$\binom{n-k}{j}$ ways a possibly empty set of the unlabelled blocks in $\pi_X$, that contribute to the possibly non-existing $(j=0)$
unlabelled block in the labelled set partition $\sigma$. Furthermore we can partition the remaining $n-j$ elements in $B(n-j)$ different
ways, which gives the contribution of the labelled blocks of $\sigma$. Observe that we have to ensure in the case $k=0$,
that there is at least one labelled block after all, which gives the condition $j\le n-1$ in the sum of $P_0(n,0)$.
\end{proof}

\section{The Splitting Formula\label{sec:splittingFormula}}
This section states Theorem~\ref{thm:SymmetricSplittingFormula}, 
which is the centrepiece of this article. 
\begin{Lemma}\label{lem:pSplittingReduced}
Let $(G^1,G^2,X)$ be a $K$-splitting of $(G,K)$. Then
\begin{align}
R(G,K)&=\left[\Z^T\pvec(G^1)\right]_0^T
\T_0
\left[
\Z^T\pvec(G^2)
\right]_0.
\end{align}
\end{Lemma}

\begin{proof}
By Theorem~\ref{thm:PsplittingFormula} we have 
\begin{align*}
R(G,K)&=\pvec(G^1)^T\M\pvec(G^2)
\end{align*}
and the factorisation of the $\M$ matrix yields
\begin{align*}
R(G,K)&=
\pvec(G^1)^T\Z\T\Z^T\pvec(G^2).
\end{align*}
Considering only the non-vanishing elements of the diagonal matrix $\T$ gives then
\begin{align*}
R(G,K)&=
\left[
\Z^T\pvec(G^1)
\right]_0^T
\T_0
\left[
\Z^T\pvec(G^2)
\right]_0. \qedhere
\end{align*}
\end{proof}

\begin{Lemma}\label{lem:rpReduced}
Let $(G^1,G^2,X)$ be a $K$-splitting of $(G,K)$. Then
\begin{align}
\left[\Z^T\pvec(G^i)\right]_0&=
\T_0^{-1}\Z^{-1}_0\rvec_0(G^i) \quad \text{for }i=1,2.
\end{align}
\end{Lemma}
\begin{proof}
By Theorem~\ref{thm:cformula} we have the equation
\begin{align*}
\M \pvec(G^i)&=\rvec(G^i)
\end{align*}
and the factorisation of the $\M$ matrix yields
\begin{align*}
\T \Z^T\pvec(G^i)&=\Z^{-1}\rvec(G^i).
\end{align*}
Considering only the non-vanishing elements of the diagonal matrix $\T$ gives
\begin{align*}
\T_0 \left[\Z^T\pvec(G^i)\right]_0&=\left[\Z^{-1}\rvec(G^i)\right]_0.
\end{align*}
Note that we have the equality 
\begin{align*}
\left[\Z^{-1}\rvec(G^i)\right]_0=\Z^{-1}_0\rvec_0(G^i).
\end{align*}
Hence we can conclude
\begin{align*}
\T_0 \left[\Z^T\pvec(G^i)\right]_0&=\Z^{-1}_0\rvec_0(G^i), 
\end{align*}
which leads after multiplication with $\T_0^{-1}$ to the desired result.
\end{proof}

\begin{Theorem}\label{thm:SymmetricSplittingFormula}
Let $(G^1,G^2,X)$ be a $K$-splitting of $(G,K)$. Then
\begin{align}
R(G,K)&=\rvec_0(G^1)^T\M_0^{-1}\rvec_0(G^2).
\end{align}
\end{Theorem}
\begin{proof}
By Lemma~\ref{lem:pSplittingReduced} we have
\begin{align*}
R(G,K)&=\left[\Z^T\pvec(G^1)\right]_0^T\T_0\left[\Z^T\pvec(G^2)\right]_0,
\end{align*}
whereas Lemma~\ref{lem:rpReduced} gives 
\begin{align*}
\left[\Z^T\pvec(G^i)\right]_0&=\T_0^{-1}\Z_0^{-1}\rvec_0(G^i)
\end{align*}
for $i=1,2$. Substituting the second equation in the first one proves the claim
\begin{align*}
R(G,K)&=\rvec_0(G^1)^T\Z_0^{-1T}\T_0^{-1T} \T_0 \T_0^{-1}\Z_0^{-1}\rvec_0(G^2) \\
&=\rvec_0(G^1)^T[\Z_0^T\T_0\Z_0]^{-1}\rvec_0(G^2) \\
&=\rvec_0(G^1)^T\M_0^{-1}\rvec_0(G^2). \qedhere
\end{align*}
\end{proof}
       
\section{Conclusion}
The splitting formula stated by Theorem~\ref{thm:SymmetricSplittingFormula} shows that the $K$-terminal
reliability can be computed for graphs with small separating vertex sets. We showed that our approach is 
superior to former known methods by achieving a tremendous reduction of the necessary 
states by utilising the factoring of the transfer matrix. 

Even more we proved that the computational efficiency can be further improved by using separating vertex sets
containing terminal vertices. 

Observe that we could easily extend our approach to a recursive decomposition scheme by the
transfer-matrix method. This extension leads to a polynomial time algorithm for the computation of $R(G,K)$ for 
graphs with restricted treewidth.

Finally, the splitting approach has an amendable form, as it expresses the $K$-terminal reliability as
a sum of linear combinations of $K$-terminal reliabilities of the subgraphs of the splitting.
            
\section{Acknowledgements}
The author wants to thank his supervisor Peter Tittmann for useful hints during
his research.

This work is supported by the ESF grant {\tt 080942497} from the European Union.

\end{document}